\documentclass[preprint,11pt]{article}

\RequirePackage[colorlinks,citecolor=blue,urlcolor=blue]{hyperref}
\usepackage{amssymb}
\usepackage{mathrsfs}
\usepackage{graphicx}
\usepackage{colortbl,dcolumn}
\usepackage{tabularx}
\usepackage{amsmath}
\usepackage{psfrag}
\usepackage{booktabs}
\usepackage{cite}
\usepackage{amsthm}
\usepackage{float}
\usepackage{subfigure}
   
\numberwithin{equation}{section}
\usepackage[top=1.2in, bottom=1.2in, left=1in, right=1in]{geometry}

\newcommand{\R}{\mathbb{R}}
\newcommand{\N}{\mathbb{N}}
\newcommand{\E}{\mathbb{E}}
\renewcommand{\P}{\mathbb{P}}

\newtheorem{tm}{Theorem}[section]
\newtheorem{df}{Definition}

\newtheorem{cor}{Corollary}
\newtheorem{rk}{Remark}
\newtheorem{example}{Example}

\newtheorem{scheme}{\textbf{Scheme}}

\allowdisplaybreaks \allowdisplaybreaks[4]

\begin{document}


\title{\bf Stochastic K-symplectic integrators for stochastic non-canonical Hamiltonian systems and applications to the Lotka--Volterra model}

       \author{
        {Jialin Hong\footnotemark[1], Lihai Ji\footnotemark[2], Xu Wang\footnotemark[3], and Jingjing Zhang\footnotemark[4]}\\
      \\
        {\rm\small \footnotemark[1]~\footnotemark[3] LSEC, ICMSEC, Academy of Mathematics and Systems Science,}\\{\rm\small Chinese Academy of Sciences, Beijing 100190, China}\\
        {\rm\small \footnotemark[1]~\footnotemark[3] School of Mathematical Sciences, University of Chinese Academy of Sciences,}\\{\rm\small Beijing 100049, China}\\
      {\rm\small \footnotemark[2] Institute of Applied Physics and Computational Mathematics, Beijing 100094, China }\\
{\rm\small \footnotemark[4] School of Science, East China Jiaotong  University, Nanchang 330013, China}
      }
       \maketitle
       \footnotetext{J. Hong and X. Wang are supported by the National Natural Science Foundation of China (No.91530118, No.91130003, No.11021101, No.91630312 and No.11290142). L. Ji is supported by the National Natural Science Foundation of China (No.11601032, No.11471310). J. Zhang is supported by  the National Natural Science Foundation of China (No.11761033).
}
        \footnotetext{\footnotemark[3]Corresponding author: wangxu@lsec.cc.ac.cn.}

       \begin{abstract}
       We give a theoretical framework of stochastic non-canonical Hamiltonian systems as well as their modified symplectic structure
which is named stochastic K-symplectic structure. The framework can be applied to the study of the Lotka--Volterra model perturbed by external noises. In terms of internal properties of the stochastic Lotka--Volterra model, we propose different kinds of stochastic K-symplectic integrators which could inherit the positivity of the solution. The K-symplectic conditions are also obtained to ensure that the proposed schemes admit the same geometric structure as the original system. Besides, the first-order condition of the proposed schemes in $L^1(\Omega)$ sense are given based on the uniform boundedness of both the exact solution and the numerical one.
Several numerical examples are illustrated to verify above properties of proposed schemes compared with non-K-symplectic ones.


\end{abstract}

\section{Introduction}\label{1.1}
The study of Hamiltonian systems for both deterministic case ${\rm d}z=J^{-1}\nabla_zH_0(z){\rm d}t$ and stochastic case ${\rm d}z=J^{-1}\nabla_zH_0(z){\rm d}t+J^{-1}\nabla_zH_1(z)\circ {\rm d}W$
on its geometric structure, long-time behaviors and numerical integrations is always of great interest (see \cite{Bismut81,BR01,HLW10,JWH03,MPS98,MRT02,MT02,SW17,W07,WHX17} and references therein).
It is well known that the following system
$${\rm d}z=J^{-1}\nabla_zH_0(z){\rm d}t+J^{-1}\nabla_zH_1(z)\circ {\rm d}W$$ can be characterized by ${\rm d}z_i=\{z_i,H_0\}{\rm d}t+\{z_i,H_1\}\circ {\rm d}W$ with canonical Poisson bracket
\begin{align*}
\{F,G\}(z):=\nabla F(z)^\top J^{-1}\nabla G(z),\quad\forall~F,G\in \mathbf{C}^1(\R^{2d}).
\end{align*}
There are also many circumstances that the behavior of the system is not applicable to be described by a stochastic Hamiltonian systems (SHS). Instead, stochastic  non-canonical Hamiltonian systems (or {\it Poisson systems})
$${\rm d}z=B(z)\nabla_z H_0(z){\rm d}t+B(z)\nabla_z H_1(z)\circ {\rm d}W$$ will be taken into consideration with an associated generalization of canonical Poisson bracket
\begin{align*}
\{F,G\}(z)= \nabla F(z)^\top B(z)\nabla G(z),
\end{align*}
where $B(z)$ are skew symmetric matrices for all $z\in\R^{2d}$.
Non-canonical Hamiltonian systems are widely used in
plasma physics, mathematical biology and so on.
For instance, stochastic Lotka--Volterra (LV) equations (see \cite{Ru03} or Example \ref{Ex1}) in two or higher dimensions, describing the dynamics of biological systems with several interacting species, are non-canonical Hamiltonian systems.
And also, the stochastic Ablowitz--Ladik discrete nonlinear Schr\"odinger equation (see \cite{Garnier} or Example \ref{Ex2}),
which is used to model the self trapping on a dimer, dynamics of anharmonic lattices and pulse dynamics in nonlinear optics with random perturbations, possesses a non-canonical symplectic structure.

For deterministic non-canonical Hamiltonian systems, their geometric structures and numerical integrations have been  well-investigated (see \cite{Awane92,Lucas15,S99} and references therein). In this paper, we study the following stochastic non-canonical Hamiltonian system \begin{equation}\label{non_canonical}
{\rm d}z=K^{-1}(z)\nabla_z H_{0}(z){\rm d}t+K^{-1}(z)\nabla_z H_{1}(z)\circ {\rm d}W(t),~~z(0)=z_0\in \mathbb{R}^{2d},
\end{equation}
where $H_0,~H_1$ are two smooth Hamiltonian functions. The $\circ$ in the last term of \eqref{non_canonical} indicates that the stochastic integral is defined in Stratonovich sense, and $W(t)$ is a one-dimensional standard Wiener process defined on the filtered probability space $(\Omega,\mathcal{F},\P;\{\mathcal{F}_t\}_{t\ge0})$. Matrices $K(z)\in\R^{2d\times 2d}$ are skew-symmetric for all $z\in\R^{2d}$ and satisfy the Jacobi identity
\begin{equation}\label{Jacobi_identity}
\frac{\partial k_{ij}}{\partial z_k}+\frac{\partial k_{jk}}{\partial z_i}+\frac{\partial k_{ki}}{\partial z_j}=0,~~{\rm for~all}~~i,j,k=1,2,\cdots,2d,
\end{equation}
where $k_{ij}:=k_{ij}(z)$ are elements of $K(z)$.

The goal of this work is to investigate the geometric structure and construct structure-preserving numerical methods for system \eqref{non_canonical} and derive the corresponding order condition. We show that system \eqref{non_canonical} under condition \eqref{Jacobi_identity} possesses the stochastic K-symplectic structure. However, it is a difficult task to construct stochastic K-symplectic integrators for general stochastic non-canonical systems due to the arbitrariness of non-constant $K(z)$. 
Therefore, special structures of specific problems have to be exploited to gain stochastic K-symplectic integrators. 
In this paper, we take the stochastic LV model as the keystone to construct Runge--Kutta type and partitioned Runge--Kutta type stochastic K-symplectic integrators by utilizing Darboux transformation technique. We prove that the proposed methods could inherit the stochastic K-symplectic structure as well as the property of the original problem, e.g., the positivity of the solution. Furthermore, based on the uniform boundedness of numerical solutions, we derive the global order conditions in $L^{1}(\Omega)$ of proposed numerical methods. Numerical experiments show the favorable performance of stochastic K-symplectic integrators compared with non-K-symplectic methods. Particularly, the similar fashion of the construction of stochastic K-symplectic integrators in this paper can be applied to other specific stochastic non-canonical Hamiltonian systems. To the best of our knowledge, there has been no work in the literature which studies stochastic K-symplectic structure or stochastic K-symplectic integrators for stochastic non-canonical Hamiltonian systems.

The rest of this paper is organized as follows. Section 2 is devoted to show the geometric structure---K-symplectic structure---of non-canonical Hamiltonian systems, and give some concrete and widely used models for this kind of systems. In Section 3, the definition of stochastic K-symplectic integrators is firstly presented. We take stochastic LV model as the keystone to construct stochastic K-symplectic integrators, and derive the convergence order conditions. In Section 4, several specific stochastic K-symplectic integrators of Runge--Kutta or partitioned Runge--Kutta type are given. Finally, numerical experiments are performed to testify the effectiveness of the proposed schemes in Section 5.

\section{Stochastic K-symplectic structure}

In this section, we introduce the stochastic K-symplectic geometric structure as well as its equivalent form for stochastic non-canonical Hamiltonian systems.
Two concrete examples are also given to show the transformation between specific systems and stochastic non-canonical Hamiltonian systems with K-symplectic structure.

It can be verified that if we choose $K(z)$ as the standard symplectic matrix $J_{2d}$, i.e.,
\begin{align*}
K(z)=
\left(\begin{array}{cc}
  0&I_d\\
  -I_d&0
\end{array}\right)=:J_{2d}
\end{align*}
with $d$-dimensional identity matrix $I_d$, system \eqref{non_canonical} will degenerate to the following canonical SHS
\begin{align}\label{canonical}
{\rm d}z=J^{-1}_{2d}\nabla_zH_0(z)dt+J^{-1}_{2d}\nabla_zH_1(z)\circ{\rm d}W(t),~~z(0)=z_0\in\R^{2d}.
\end{align}
The phase flow $\phi_t:z_0\mapsto z(t)$ of \eqref{canonical} preserves the natural stochastic symplectic structure of the phase space :
\begin{equation*}
\left[\frac{\partial \phi_t(z_0)}{\partial z_0}\right]^{\top}J_{2d}\left[\frac{\partial \phi_t(z_0)}{\partial z_0}\right]=J_{2d},\quad \P\text{-a.s.}
\end{equation*}
which is a characteristic property of a SHS and is known as the stochastic symplectic conservation law.
Equivalently, the stochastic symplectic conservation law can also be described with differential 2-forms:
\begin{equation*}
{\rm d}P\wedge{\rm d}Q={\rm d}p\wedge{\rm d}q,\quad\P\text{-a.s.}
\end{equation*}
by setting $z=(P^\top,Q^\top)^{\top}$ with $P,Q\in\R^d$.
While for the non-canonical case, the standard symplectic conservation law is not preserved anymore. A generalized version of symplectic geometric structure has to be taken into consideration.

In the following, equations with respect to random variables all hold almost surely with respect to the probability measure $\P$, and we omit
the notation `$\P$-a.s.' for convenience unless it is necessary.

\begin{df}\label{def_1}
  Let $U\subset\mathbb{R}^{2d}$ be an open set. A differentiable map $f:U\rightarrow \mathbb{R}^{2d}$ is called stochastic K-symplectic if for all $z\in U$, the Jacobian matrix $\nabla f$ satisfies
  \begin{equation*}\label{map}
    \nabla f(z)^{\top}K(f(z))\nabla f(z)=K(z).
  \end{equation*}
\end{df}

Recall that the phase flow $\varphi_t$ of \eqref{non_canonical} is the mapping that advances the solution by time $t$, i.e.,
\begin{equation*}
  \varphi_t:~z_0\mapsto z(t),
\end{equation*}
where $z(t)$ is the solution of \eqref{non_canonical} corresponding to the initial value $z(0)=z_0$.
We show that $\varphi_t$ is K-symplectic as in the definition above. Furthermore, the K-symplectic structure also has an equivalent form in terms of differential 2-forms.

\begin{tm}\label{tm_1}
Let $H_0(z)$ and $H_1(z)$ in \eqref{non_canonical} be two twice continuous differentiable functions on $U\subset\mathbb{R}^{2d}$, and condition \eqref{Jacobi_identity} hold. Then, the flow $\varphi_t$ of \eqref{non_canonical}, starting from $z_0$, is a stochastic K-symplectic transformation, and possesses the K-symplectic conservation law
\begin{equation}\label{stochastic_k_symplectic}
\left[\frac{\partial \varphi_t(z_0)}{\partial z_0}\right]^{\top}K(\varphi_t(z_0))\left[\frac{\partial \varphi_t(z_0)}{\partial z_0}\right]=K(z_0).
\end{equation}
\end{tm}
The proof of this theorem as well as the following one is given in the appendix such that the main results are stated in a clear fashion.
\begin{tm}\label{tm_2}
  For stochastic non-canonical Hamiltonian system \eqref{non_canonical} with condition \eqref{Jacobi_identity}, it holds
  \begin{equation}\label{wedge_theorem}
    {\rm d}z(t)\wedge K(z(t)){\rm d}z(t)={\rm d}z_0\wedge K(z_0){\rm d}z_0
  \end{equation}
  for all $t\geq0$.
\end{tm}

The formulation of K-symplectic structure in Theorem \ref{tm_1} and \ref{tm_2} are equivalent, which is stated in the following corollary.

\begin{cor}
Equation \eqref{wedge_theorem} is equivalent to \eqref{stochastic_k_symplectic}.
\end{cor}
\begin{proof}
Substituting $z=\varphi_t(z_0)$ into \eqref{stochastic_k_symplectic}, it yields
\begin{equation*}
\begin{split}
  {\rm d}z\wedge K(z(t)){\rm d}z&=\frac{\partial \varphi_t(z_0)}{\partial z_0}{\rm d}z_0\wedge K(z(t))\frac{\partial \varphi_t(z_0)}{\partial z_0}{\rm d}z_0\\[3mm]
  &={\rm d}z_0\wedge \left(\left[\frac{\partial \varphi_t(z_0)}{\partial z_0}\right]^{\top}K(\varphi_t(z_0))\left[\frac{\partial \varphi_t(z_0)}{\partial z_0}\right]\right){\rm d}z_0.
\end{split}
\end{equation*}
The necessity follows from Theorem \ref{tm_1}. For the proof of sufficiency, we know that
\begin{equation*}
  {\rm d}z_0\wedge \left(\left[\frac{\partial \varphi_t(z_0)}{\partial z_0}\right]^{\top}K(\varphi_t(z_0))\left[\frac{\partial \varphi_t(z_0)}{\partial z_0}\right]\right){\rm d}z_0={\rm d}z_0\wedge K(z_0){\rm d}z_0,
\end{equation*}
i.e.,
\begin{equation*}
  {\rm d}z_0\wedge \left(\left[\frac{\partial \varphi_t(z_0)}{\partial z_0}\right]^{\top}K(\varphi_t(z_0))\left[\frac{\partial \varphi_t(z_0)}{\partial z_0}\right]-K(z_0)\right){\rm d}z_0=0.
\end{equation*}
Due to the fact that $K$ is a skew-symmetric matrix, it follows from the wedge property ${\rm d}z_0\wedge B{\rm d}z_0=0$  if and only if $B^{\top}=B$ that
\begin{equation*}
 \left[\frac{\partial \varphi_t(z_0)}{\partial z_0}\right]^{\top}K(\varphi_t(z_0))\left[\frac{\partial \varphi_t(z_0)}{\partial z_0}\right]=K(z_0),
\end{equation*}
which complete the proof.
\end{proof}

The following two examples are typical models which possess stochastic K-symplectic structure.
\begin{example}\label{Ex1}
{{\bf Stochastic LV Model}}
\vspace{2mm}

Consider the following classical stochastic LV predator-prey model (see e.g. \cite{Ru03})
\begin{equation*}
 \begin{split}
   {\rm d}x(t)&=x(t)(-\gamma_2y(t)+\eta_2){\rm d}t+\sigma_2 x(t)\circ{\rm d}W(t),\\
   {\rm d}y(t)&=y(t)(\gamma_1x(t)-\eta_1){\rm d}t+\sigma_1 y(t)\circ{\rm d}W(t)
 \end{split}
\end{equation*}
with constants $\sigma_1,\sigma_2$, positive constants $\gamma_i,\eta_i~(i=1,2)$, and $x(t)$ and $y(t)$ representing populations of the prey and the predator respectively. Let
\begin{align*}
H_0(x,y)=&-\gamma_1x(t)+\eta_1\ln x(t)-\gamma_2y(t)+\eta_2\ln y(t),\\
H_1(x,y)=&-\sigma_1\ln x(t)+\sigma_2\ln y(t).
\end{align*}
We have
\begin{equation*}
  {\rm d}\left(\begin{array}{c}
  x\\
  y
  \end{array}\right)=\left(\begin{array}{cc}
  0&xy\\
  -xy&0
  \end{array}\right)\nabla H_0(x,y){\rm d}t+\left(\begin{array}{cc}
  0&xy\\
  -xy&0
  \end{array}\right)\nabla H_1(x,y)\circ{\rm d}W(t),
\end{equation*}
which implies that the stochastic LV model \eqref{stochastic_LV_model} is a stochastic non-canonical Hamiltonian system.
\end{example}
\begin{example}\label{Ex2}
{{\bf Stochastic Ablowitz--Ladik Model}}
\vspace{2mm}

Another example is the stochastic Ablowitz--Ladik model, which describes the transmission of solitons in discrete, nonlinear and random media. It can also be regard as a spacial discretization of the stochastic nonlinear Schr\"{o}dinger equation:
\begin{equation*}
  {\bf i}{\rm d}u_k(t)+\frac{1}{\Delta x^2}(u_{k+1}-2u_{k}+u_{k-1}){\rm d}t+|u_k|^{2}\frac{u_{k+1}+u_{k-1}}{2}{\rm d}t=u_k\circ {\rm d}W(x_k,t),~k=1,\cdots,M
\end{equation*}
under periodic boundary conditions $u_{k+M}=u_{k}$, $x_k=k\Delta x$ and $\Delta x=1/M$. Here, $W$ is a real-valued $Q$-Wiener process.

By splitting the variables into real and imaginary parts, $u_k(t)=p_k(t)+{\bf i}q_k(t),~k=1,2,\cdots,M$, we obtain
\begin{equation*}
  \begin{split}
    {\rm d}p_k(t)&=-\frac{1}{\Delta x^2}(q_{k+1}-2q_{k}+q_{k-1}){\rm d}t-(p_k^2+q_k^2)\frac{q_{k+1}+q_{k-1}}{2}{\rm d}t+q_k\circ {\rm d}W(x_k,t),\\[2mm]
    {\rm d}q_k(t)&=~~\frac{1}{\Delta x^2}(p_{k+1}-2p_{k}+p_{k-1}){\rm d}t+(p_k^2+q_k^2)\frac{q_{k+1}+q_{k-1}}{2}{\rm d}t-p_k\circ {\rm d}W(x_k,t).
  \end{split}
\end{equation*}
Let $P=(p_1,\cdots,p_M),~Q=(q_1,\cdots,q_M)$ and $Z=(P,Q)^{\top}$, this system turns to be
\begin{equation*}
  {\rm d}Z=\left(\begin{array}{cc}
  0&-D(Z)\\
  D(Z)&0
  \end{array}\right)\nabla H_0(Z){\rm d}t+\left(\begin{array}{cc}
  0&-D(Z)\\
  D(Z)&0
  \end{array}\right)\nabla H_1(Z)\circ{\rm d}W(x_k,t),
\end{equation*}
where $D(Z)={\rm diag}\left(d_1(Z),\cdots,d_M(Z)\right)$ is a diagonal matrix with entries
\begin{equation*}
  d_k(Z)=1+\frac{\Delta x^2}{2}(p_k^2+q_k^2),~k=1,2,\cdots,M
\end{equation*}
and the Hamiltonian functions are
\begin{equation*}
  \begin{split}
    H_0(Z)&=\frac{1}{\Delta x^2}\sum_{k=1}^{M}(p_kp_{k-1}+q_kq_{k-1})-\frac{2}{\Delta x^4}\sum_{k=1}^{M}\ln\Big(1+\frac{\Delta x^2}{2}(p_k^2+q_k^2)\Big),\\[2mm]
    H_1(Z)&=-\frac{1}{\Delta x^2}\sum_{k=1}^{M}\ln\Big(1+\frac{\Delta x^2}{2}(p_k^2+q_k^2)\Big).
  \end{split}
\end{equation*}

\end{example}

\section{Stochastic K-symplectic methods for stochastic LV model}\label{sec3}
Since the phase flow of \eqref{non_canonical} preserves stochastic K-symplectic structure \eqref{wedge_theorem}, one wants to construct numerical methods which inherit this property.
\begin{df}\label{df_k_symplectic_method}
  A numerical method with solution $\{z_n\}_{n\ge1}$ for solving \eqref{non_canonical} is said to be stochastic K-symplectic if
  \begin{equation}\label{df_2}
    {\rm d}z_{n+1}\wedge K(z_{n+1}){\rm d}z_{n+1}={\rm d}z_{n}\wedge K(z_{n}){\rm d}z_{n}
      \end{equation}
  for all $n\geq1$.
\end{df}

In this section, we take the stochastic LV model as the keystone, which plays an important role in biological populations, ecology and mathematical ecology, to illustrate the procedure of constructing stochastic K-symplectic integrators. Consider now the following concrete class of stochastic LV model in Stratonovich sense
\begin{equation}\label{stochastic_LV_model}
 \begin{split}
   {\rm d}x(t)&=x(t)(-\gamma_2y(t)+\eta_2){\rm d}t+\sigma_2 x(t)\circ{\rm d}W(t),\\
   {\rm d}y(t)&=y(t)(\gamma_1x(t)-\eta_1){\rm d}t+\sigma_1 y(t)\circ{\rm d}W(t)
 \end{split}
\end{equation}
starting from $x(0)=x_0$ and $y(0)=y_0$.
\subsection{Properties of stochastic LV model}

We choose the state space as $\R^2$ for simplicity, i.e., $(x,y)^\top\in\R^2$.
The global well-posedness and positivity of the solution to \eqref{stochastic_LV_model} is stated in the following theorem.

\begin{tm}\label{positive}
For any deterministic initial value $(x_0,y_0)^{\top}\in\R^{2}_+:=(0,+\infty)\times(0,+\infty)$, if the coefficients in \eqref{stochastic_LV_model} satisfy
\begin{align}\label{coefficients}
\gamma_1=\gamma_2=\eta_1-\frac{\sigma_1^2}2=\eta_2+\frac{\sigma_2^2}2,
\end{align}
then \eqref{stochastic_LV_model} admits a unique solution $(x(t),y(t))$. Moreover, $(x(t),y(t))^{\top}\in\R^2_+$ for all $t\ge0$.
\end{tm}
\begin{proof}

Since the initial value $(x_0,y_0)^{\top}\in\R_+^2$, there must exists some $m_0\in\N_+$ such that $x_0,y_0\in[\frac1{m_0},m_0]$.
To show the global well-posedness of the solution, we first denote the explosion time by
$$\tau_e:=\inf\left\{t>0|\;|x(t)|=\infty\;\text{or}\;|y(t)|=\infty\right\},$$
before which the solution starting from $(x_0,y_0)^{\top}$ won't blow up. In addition, for any integer $m\ge m_0$, we define stopping times
\begin{align*}
\tau_m=\inf\left\{t\in[0,\tau_e)\bigg{|}\;x(t)\notin\left(\frac1m,m\right)\;\text{or}\;y(t)\notin\left(\frac1m,m\right)\right\}.
\end{align*}
It then suffices to show that $\tau_{\infty}:=\lim_{m\to\infty}\tau_m=\infty$ almost surely.

Motivated by \cite{LJM09}, we assume by contradiction that $\tau_\infty<\infty$ with positive probability. More precisely, there exist some constants $T_0>0$, $\varepsilon\in(0,1)$ and $m_1>m_0$ such that  $\P(\tau_m\le T_0)>\varepsilon$ for all $m\ge m_1$.
We introduce the generator (see e.g. \cite{S94}) $\mathcal{L}$ of the equivalent It\^o form of \eqref{stochastic_LV_model}:
\begin{align*}
\mathcal{L}:=&\left(-\gamma_2xy+\left(\eta_2+\frac{\sigma_2^2}2\right)x\right)\frac{\partial}{\partial x}+\left(\gamma_1xy-\left(\eta_1-\frac{\sigma_1^2}2\right)y\right)\frac{\partial}{\partial y}\\
&+\frac{\sigma_2^2x^2}2\frac{\partial^2}{\partial x^2}+\sigma_1\sigma_2 xy\frac{\partial^2}{\partial x\partial y}+\frac{\sigma_1^2y^2}2\frac{\partial^2}{\partial y^2}
\end{align*}
and an auxiliary function $V(x,y):=(x-1-\ln x)+(y-1-\ln y)$ defined on $\R^2_+$. Simple calculation yields that
$\mathcal{L}V(x,y)\equiv(\sigma_1^2+\sigma_2^2)/2$ for $(x,y)^{\top}\in\R_+^2$ based on the condition \eqref{coefficients}. Hence, we get
\begin{align*}
\E V(x(\tau_m\wedge T_0),y(\tau_m\wedge T_0))=&V(x_0,y_0)+
\E\int_0^{\tau_m\wedge T_0}\mathcal{L}V(x(t),y(t))dt\\
\le&V(x_0,y_0)+\frac{\sigma_1^2+\sigma_2^2}2T_0<\infty,
\end{align*}
where in the last step we have used the fact that for any $t\in(0,\tau_m\wedge T_0)$, $x(t),y(t)\in(\frac1m,m)$ and thus $\mathcal{L}V(x(t),y(t))\equiv(\sigma_1^2+\sigma_2^2)/2$.

Note that function $V$ has a minimum at $(1,1)^{\top}$ with $V(1,1)=0$, and is decreasing on $(0,1)$ and increasing on $(1,+\infty)$. Also, for any $m\ge m_1$ and any sample $\omega\in\{\tau_m\le T_0\}$, either $x(\tau_m)$ or $y(\tau_m)$ reaches the boundary of interval $\left(\frac1m,m\right)$. Thus, we obtain
\begin{align*}
V(x(\tau_m,\omega),y(\tau_m,\omega))
\ge\left(\frac1m-1-\log \frac1m\right)\wedge\left(m-1-\log m\right).
\end{align*}
We conclude from all the above that
\begin{align*}
V(x_0,y_0)+\frac{\sigma_1^2+\sigma_2^2}2T_0\ge&\E V(x(\tau_m\wedge T_0),y(\tau_m\wedge T_0))
\ge\E\left[ V(x(\tau_m),y(\tau_m))\mathbf{1}_{\{\tau_m\le T_0\}}\right]\\
\ge&\left(\frac1m-1-\log \frac1m\right)\wedge\left(m-1-\log m\right)\epsilon
\to\infty
\end{align*}
as $m\to\infty$, which is a contradiction with $V(x_0,y_0)+\frac{\sigma_1^2+\sigma_2^2}2T_0<\infty$. Hence, $\tau_\infty=\infty$ and we complete the proof.
\end{proof}

The $p$th moment of the exact solution are uniformly bounded.

\begin{cor}\label{cor3.1}
For any $(x_0,y_0)^{\top}\in\R^2_+$ and $p\ge1$, the $p$th moment of solution $(x(t),y(t))^{\top}$ with $t\in[0,T]$ is uniformly bounded under conditions given above, i.e., there exists a positive constant $C=C(x_0,y_0,p,T)$ such that
\begin{align*}
\sup_{t\in[0,T]}\E\left[(x(t))^p+(y(t))^{p}\right]\le\sup_{t\in[0,T]}\E\left[\left(x(t)+y(t)\right)^{p}\right]\le C.
\end{align*}
\end{cor}
\begin{proof}
This proof is given through a similar approach to that in \cite{LJM09}.
Define another auxiliary function $F_p(t,x,y)=e^t(x+y)^p>0$ on the domain $[0,T]\times\R_+^2$. It is worth noticing that
\begin{align*}
\mathcal{L}F_p(t,x,y)=&\gamma_1(-xy+x+xy-y)e^tp(x+y)^{p-1}\\
&+\left(\frac{\sigma_2^2x^2}2+\sigma_1\sigma_2 xy+\frac{\sigma_1^2y^2}2\right)e^tp(p-1)(x+y)^{p-2}\\
\le&Cp^2F_p(t,x,y),
\end{align*}
in which the fact $x-y\le x+y$ and $\frac{\sigma_2^2x^2}2+\sigma_1\sigma_2 xy+\frac{\sigma_1^2y^2}2\le \frac{|\sigma_1|^2\vee|\sigma_2|^2}2(x+y)^2$ for any $(x,y)^{\top}\in\R_+^2$ are used.
According to Theorem \ref{positive}, we have $(x(t),y(t))^{\top}\in\R_+^2$ for any $t\ge0$.
Hence, It\^o's formula applied to $F_p(t,x(t),y(t))$ indicates that
\begin{align*}
\E[ F_p(t,x(t),y(t))]=&F_p(0,x_0,y_0)+\E\int_0^t\frac{\partial }{\partial s}F_p(s,x(s),y(s))ds\\
&+\E\int_0^t\mathcal{L}F_p(s,x(s),y(s))ds\\
\le&\left(x_0+y_0\right)^p+\left(1+Cp^2\right)\int_0^t\E[F_p(s,x(s),y(s))]ds.
\end{align*}
Multiplying both sides of above inequality by $e^{-t}$, we have
\begin{align*}
\E[(x(t)+y(t))^p]\le e^{-t}(x_0+y_0)^p+(1+Cp^2)\int_0^te^{-(t-s)}\E[(x(s)+y(s))^p]ds,
\end{align*}
which, together with Gronwall's inequality, shows
\begin{align*}
\sup_{t\in[0,T]}\E[(x(t)+y(t))^p]\le\sup_{t\in[0,T]}
\left[(x_0+y_0)^pe^{-t}+\frac1{Cp^2}\left(e^{Cp^2t}-1\right)\right]
\le \tilde C.
\end{align*}
We then complete the proof.
\end{proof}

\subsection{Stochastic K-symplectic integrators}

We use the results of the previous section to construct several classes of stochastic K-symplectic integrators for stochastic LV model \eqref{stochastic_LV_model} and derive the convergence order conditions.
\subsubsection{Stochastic K-Runge--Kutta methods}

Based on the positivity of the solution and introducing auxiliary coordinates $x=e^u$ and $y=e^v$, \eqref{stochastic_LV_model} has an equivalent standard SHS form
\begin{equation}\label{uv}
  \begin{split}
    &du=\left(-\gamma_2e^v+\eta_2\right)dt+\sigma_2dW,\\
    &dv=\left(\gamma_1e^u-\eta_1\right)dt+\sigma_1dW.
   \end{split}
\end{equation}

Applying the Runge--Kutta methods to above equations and utilizing the inverse transformation $u=\ln x, v=\ln y$, we finally get what we call $s$-stage stochastic K-Runge--Kutta (KRK) methods:
\begin{equation}\label{sto_RK}
  \begin{split}
    &x_i=X_n\exp\Big(h\sum_{j=1}^{s}a_{ij}(-\gamma_2y_j+\eta_2)\Big)\exp\Big(\sigma_2J_n\sum_{j=1}^sb_{ij}\Big),~ i=1,\cdots,s,\\
    &y_i=Y_n\exp\Big(h\sum_{j=1}^{s}a_{ij}(\gamma_1x_j-\eta_1)\Big)\exp\Big(\sigma_1J_n\sum_{j=1}^{s}b_{ij}\Big),\\
    &X_{n+1}=X_n\exp\Big(h\sum_{j=1}^{s}\alpha_{j}(-\gamma_2y_j+\eta_2)\Big)\exp\Big(\sigma_2J_n\sum_{j=1}^s\beta_j\Big),~n=0,\cdots,N-1,\\
    &Y_{n+1}=Y_n\exp\Big(h\sum_{j=1}^{s}\alpha_{j}(\gamma_1x_j-\eta_1)\Big)\exp\Big(\sigma_1J_n\sum_{j=1}^{s}\beta_{j}\Big)
  \end{split}
\end{equation}
starting from $X_0=x_0$ and $Y_0=y_0$, where $s$ is an integer, $h=t_{n+1}-t_n$ is the uniform time step, and $J_{n}=\sqrt{h}\xi_n$ has the same distribution and adaptness with $W(t_{n+1})-W(t_n)$.  Here $\xi_n,n=1,\cdots,N,$ are independent random variables satisfying standard normal distribution. Without loss of generality, we assume that $N=T/h$ is an integer. These numerical methods parameters can be characterized by the Butcher tableau
\begin{center}
\begin{tabular}{c|cc}
       & $A$ & $B$\\
\hline & $\alpha^{\top}$ & $\beta^{\top}$
\end{tabular}
\end{center}
with $A=(a_{ij})_{i,j=1}^{s},~B=(b_{ij})_{i,j=1}^{s},~\alpha=(\alpha_1,\cdots,\alpha_s)^\top,~\beta=(\beta_1,\cdots,\beta_s)^\top$.

Next, we proceed to the stochastic K-symplectic conditions results.
\begin{tm}\label{KRK}
  Assume that the coefficients $a_{ij},b_{ij},\alpha_i,\beta_i$ in \eqref{sto_RK} satisfy
    \begin{equation}\label{condition}
    \begin{split}
      &\alpha_i a_{ij}+\alpha_j a_{ji}=\alpha_i\alpha_j,\\[2mm]
      &\alpha_i b_{ij}+\beta_j a_{ji}=\alpha_i\beta_j,\\[2mm]
      &\beta_i b_{ij}+\beta_j b_{ji}=\beta_i\beta_j
    \end{split}
  \end{equation}
for all $i,j=1,2,\cdots,s,$ then stochastic KRK methods \eqref{sto_RK} preserve the stochastic K-symplectic structure.
\end{tm}
\begin{proof}
Runge--Kutta methods applied to \eqref{uv} yield
\begin{equation}\label{sto_RK_sym}
  \begin{split}
    &u_i=U_n+h\sum_{j=1}^{s}a_{ij}(-\gamma_2e^{v_j}+\eta_2)+\sigma_2J_n\sum_{j=1}^{s}b_{ij},\\
    &v_i=V_n+h\sum_{j=1}^{s}a_{ij}(\gamma_1e^{u_j}-\eta_1)+\sigma_1J_n\sum_{j=1}^{s}b_{ij},\\
    &U_{n+1}=U_n+h\sum_{j=1}^{s}\alpha_{j}(-\gamma_1e^{v_j}+\eta_2)+\sigma_2J_n\sum_{j=1}^{s}\beta_{j},\\
    &V_{n+1}=V_n+h\sum_{j=1}^{s}\alpha_{j}(\gamma_1e^{u_j}-\eta_1)+\sigma_1J_n\sum_{j=1}^{s}\beta_{j}.
  \end{split}
\end{equation}
By simple calculation, it can be seen that methods \eqref{sto_RK_sym} with conditions \eqref{condition} preserve stochastic symplectic structure
\begin{equation}
{\rm d}U_{n+1}\wedge {\rm d}V_{n+1}={\rm d}U_{n}\wedge {\rm d}V_{n}.
\end{equation}
If we transform back to the original variables, we obtain
\begin{equation*}
  {\rm d}U_{n}\wedge {\rm d}V_{n}=\frac{1}{X_nY_n}{\rm d}X_n\wedge{\rm d}Y_n.
\end{equation*}
Thus, it yields
\begin{equation}
  \frac{1}{X_{n+1}Y_{n+1}}{\rm d}X_{n+1}\wedge{\rm d}Y_{n+1}=\frac{1}{X_nY_n}{\rm d}X_n\wedge{\rm d}Y_n,
\end{equation}
i.e.,
\begin{equation}
  {\rm d}z_{n+1}\wedge K(z_{n+1}){\rm d}z_{n+1}={\rm d}z_n\wedge K(z_n){\rm d}z_n,
\end{equation}
where $z_n=(X_n,Y_n)^{\top}$, and
\begin{equation*}
  K(z_n)=\left(\begin{array}{cc}
  0&-\frac{1}{X_{n}Y_{n}}\\[2mm]
  \frac{1}{X_{n}Y_{n}}&0
  \end{array}\right).
\end{equation*}
This completes the proof.
\end{proof}

\subsubsection{Stochastic K-Partitioned-Runge--Kutta methods}
Moreover, we consider the following numerical methods with more coefficients than those in \eqref{sto_RK}:
\begin{equation}\label{sto_PRK}
  \begin{split}
    &x_i=X_n\exp\Big(h\sum_{j=1}^{s}a_{ij}(-\gamma_2y_j+\eta_2)\Big)\exp\Big(\sigma_2J_n\sum_{j=1}^{s}b_{ij}\Big),\\
    &y_i=Y_n\exp\Big(h\sum_{j=1}^{s}\tilde{a}_{ij}(\gamma_1x_j-\eta_1)\Big)\exp\Big(\sigma_1J_n\sum_{j=1}^{s}\tilde{b}_{ij}\Big),\\
    &X_{n+1}=X_n\exp\Big(h\sum_{j=1}^{s}\alpha_{j}(-\gamma_2y_j+\eta_2)\Big)\exp\Big(\sigma_2J_n\sum_{j=1}^{s}\beta_{j}\Big),\\
    &Y_{n+1}=Y_n\exp\Big(h\sum_{j=1}^{s}\tilde{\alpha}_{j}(\gamma_1x_j-\eta_1)\Big)\exp\Big(\sigma_1J_n\sum_{j=1}^{s}\tilde{\beta}_{j}\Big),
  \end{split}
\end{equation}
where $i=1,\cdots,s$ and $n=0,\cdots,N-1$.
These methods  can be characterized through their parameters via the following Butcher tableau
\begin{center}
\begin{tabular}{c|cccc}
       & $A$ & $\tilde{A}$ & $B$ & $\tilde{B}$\\
\hline & $\alpha^{\top}$ & $\tilde{\alpha}^{\top}$ & $\beta^{\top}$ & $\tilde{\beta}^{\top}$
\end{tabular}
\end{center}
with $A=(a_{ij})_{i,j=1}^{s},~\tilde{A}=(\tilde{a}_{ij})_{i,j=1}^{s},~B=(b_{ij})_{i,j=1}^{s},~\tilde{B}=(\tilde{b}_{ij})_{i,j=1}^{s},~\alpha=(\alpha_1,\cdots,\alpha_s)^\top,~\tilde{\alpha}=(\tilde{\alpha}_1,\cdots,\tilde{\alpha}_s)^\top,~\beta=(\beta_1,\cdots,\beta_s)^\top,~\tilde{\beta}=(\tilde{\beta}_1,\cdots,\tilde{\beta}_s)^\top$.
Similarly, we can get the stochastic K-symplectic conditions for \eqref{sto_PRK}, which is stated in the following theorem.
\begin{tm}\label{KPRK1}
  Assume that the coefficients $a_{ij},\tilde{a}_{ij},b_{ij},\tilde{b}_{ij},\alpha_i,\tilde{\alpha}_i,\beta_i,\tilde{\beta}_i$ in \eqref{sto_PRK} satisfy
  \begin{equation*}
    \begin{split}
      &\alpha_i=\tilde{\alpha}_i,~\beta_i=\tilde{\beta}_i,\\[2mm]
      &\alpha_i \tilde{a}_{ij}+\tilde{\alpha}_j a_{ji}=\alpha_i\tilde{\alpha}_j,\\[2mm]
      &\beta_i \tilde{a}_{ij}+\tilde{\alpha}_j b_{ji}=\beta_i\tilde{\alpha}_j,\\[2mm]
      &\alpha_i \tilde{b}_{ij}+\tilde{\beta}_j a_{ji}=\alpha_i\tilde{\beta}_j,\\[2mm]
      &\beta_i \tilde{b}_{ij}+\tilde{\beta}_j b_{ji}=\beta_i\tilde{\beta}_j
    \end{split}
  \end{equation*}
for all $i,j=1,2,\cdots,s,$ then stochastic KPRK methods \eqref{sto_PRK} preserve stochastic K-symplectic structure.
\end{tm}
\begin{proof}
  The proof is analogous to that of the stochastic KRK case, so we ignore it here.
\end{proof}
In this paper, we call methods \eqref{sto_PRK} as $s$-stage stochastic K-Partitioned-Runge--Kutta (KPRK) methods. 
\begin{rk}
If the coefficients in \eqref{sto_PRK} satisfy
\begin{equation*}
  \alpha_i=\tilde{\alpha}_i,~\beta_i=\tilde{\beta}_i,~a_{ij}=\tilde{a}_{ij},~b_{ij}=\tilde{b}_{ij},~~i,j=1,2,\cdots,s,
\end{equation*}
 stochastic KPRK methods \eqref{sto_PRK} turns to be the same as stochastic KRK methods \eqref{sto_RK}.
\end{rk}
\subsubsection{General stochastic K-symplectic methods}
Based on the construction of stochastic K-symplectic integrators in the previous subsections, we derive the general stochastic K-symplectic methods. It is stated in the following theorem.
\begin{tm}
Suppose that there exists a transformation $v=\psi(z)$ that brings the stochastic non-canonical Hamiltonian system \eqref{non_canonical} to canonical form (as in \eqref{canonical}) and holds
\begin{equation}
  \left[\frac{\partial \psi(z)}{\partial z}\right]^{\top}J\left[\frac{\partial\psi(z)}{\partial z}\right]=K(z)
\end{equation}
with $
  J=\left(\begin{array}{cc}
  0&I_d\\
  -I_d&0
  \end{array}\right)$, then a stochastic symplectic method with solution $\{v_n\}_{n\ge1}$ gives a stochastic K-symplectic method with solution $\left\{z_n=\psi^{-1}(v_n)\right\}_{n\ge1}$.
\end{tm}
\begin{proof}
Since solution $\{v_n\}_{n\ge1}$ is symplectic, it holds
\begin{equation}
  \left[\frac{\partial v_n}{\partial v_0}\right]^{\top}J\left[\frac{\partial v_n}{\partial v_0}\right]=J.
\end{equation}
Then we have
\begin{equation}
\begin{split}
  J=&\left[\frac{\partial \psi(z_n)}{\partial z_n}\frac{\partial z_n}{\partial z_0}\frac{\partial z_0}{\partial v_0}\right]^{\top}J\left[\frac{\partial \psi(z_n)}{\partial z_n}\frac{\partial z_n}{\partial z_0}\frac{\partial z_0}{\partial v_0}\right]\\[2mm]
=&\left[\frac{\partial z_0}{\partial v_0}\right]^{\top}\left[\frac{\partial z_n}{\partial z_0}\right]^{\top}\left(\left[\frac{\partial \psi(z_n)}{\partial z_n}\right]^{\top}J\left[\frac{\partial \psi(z_n)}{\partial z_n}\right]\right)\frac{\partial z_n}{\partial z_0}\frac{\partial z_0}{\partial v_0}\\[2mm]
=&\left[\frac{\partial z_0}{\partial v_0}\right]^{\top}\left[\frac{\partial z_n}{\partial z_0}\right]^{\top}K(z_n)\left[\frac{\partial z_n}{\partial z_0}\right]\left[\frac{\partial z_0}{\partial v_0}\right],
\end{split}
\end{equation}
which yields
\begin{equation*}
\begin{split}
  \left[\frac{\partial z_n}{\partial z_0}\right]^{\top}K(z_n)\left[\frac{\partial z_n}{\partial z_0}\right]=\left[\frac{\partial z_0}{\partial v_0}\right]^{-T}J\left[\frac{\partial z_0}{\partial v_0}\right]^{-1}
=\left[\frac{\partial v_0}{\partial z_0}\right]^{\top}J\left[\frac{\partial v_0}{\partial z_0}\right]=K(z_0).
\end{split}
\end{equation*}
This indicates the K-symplectic structure associated to the solution $\{z_n\}_{n\ge1}.$
\end{proof}

\subsection{Convergence order conditions}

To avoid confusion of the coefficients, in this subsection, we choose $\gamma_1=\gamma_2=\eta_2=\sigma_1=1$, $\eta_1=\frac32$ and $\sigma_2=0$, which satisfy the condition \eqref{coefficients}, as an illustration of the following procedure, and get
\begin{equation}\label{LV}
  \begin{split}
    {\rm d}x&=\left(-xy+x\right){\rm d}t,\qquad\qquad\quad~~x(0)=x_0,\\
    {\rm d}y&=\left(xy-\frac32y\right){\rm d}t+y\circ{\rm d}W(t),~y(0)=y_0.
  \end{split}
\end{equation}
The stochastic KRK methods \eqref{sto_RK} applied to \eqref{LV} turn to be
\begin{equation}\label{sto_RK_sim}
  \begin{split}
    &x_i=X_n\exp\Big(h\sum_{j=1}^{s}a_{ij}(-y_j+1)\Big),\\
    &y_i=Y_n\exp\Big(h\sum_{j=1}^{s}a_{ij}(x_j-\frac32)\Big)\exp\Big(J_n\sum_{j=1}^{s}b_{ij}\Big),\\
    &X_{n+1}=X_n\exp\Big(h\sum_{j=1}^{s}\alpha_{j}(-y_j+1)\Big),\\
    &Y_{n+1}=Y_n\exp\Big(h\sum_{j=1}^{s}\alpha_{j}(x_j-\frac32)\Big)\exp\Big(J_n\sum_{j=1}^{s}\beta_{j}\Big)
  \end{split}
\end{equation}
with $X_0=x_0$, $Y_0=y_0$ and parameters $a_{ij},b_{ij},\alpha_j,\beta_j,i,j=1,\cdots,s$.
\begin{tm}\label{tm3.2}
Each of methods \eqref{sto_RK_sim} admits an $\mathcal{F}_{t_n}$-adapted solution $(X_n,Y_n)^\top\in\R^2_+$. Furthermore, the $p$th moment of the numerical solution is uniformly bounded
\begin{align*}
\sup_{n=1,\cdots,N}\E\left[X_n^p+Y_n^{p}\right]\le C
\end{align*}
with constant $C=C(x_0,y_0,p,T,A,\alpha,\beta)$.
\end{tm}

\begin{proof}
The existence and adaptness of the solution of \eqref{sto_RK_sim} is not difficult to get by means of the procedure given in \cite{DD04}.
Based on the fact that $x_i,y_i\in\R_+$, and denoting constants $C_\alpha=\sum_{j=1}^s\alpha_j$ and $C_\beta=\sum_{j=1}^s\beta_j$, we have
\begin{align*}
X_{n+1}\le&X_n\exp\left(\frac{h}2C_\alpha\right)\le x_0\exp\left(\frac{T}2C_\alpha\right),\\
Y_{n+1}\le&Y_n\exp\left(h\sum_{j=1}^s\alpha_jX_n\exp\left(h\sum_{k=1}^sa_{jk}\right)\right)\exp\left(J_nC_\beta\right)\\
\le&Y_n\exp\left(hx_0e^{\frac{TC_\alpha}2}\sum_{j=1}^s\alpha_j\exp\left(h\sum_{k=1}^sa_{jk}\right)\right)\exp\left(J_nC_\beta\right)\\
\le&y_0\exp(C_A)\exp\left(C_\beta\sum_{k=1}^nJ_k\right)
\end{align*}
with $C_A=Tx_0e^{\frac{TC_\alpha}2}\sum_{j=1}^s\alpha_j\exp\left(h\sum_{k=1}^sa_{jk}\right)$ for all $n=0,1,\cdots,N-1$.
In the estimation of the second inequality above, we have used the fact that $x_j\le X_n\exp(h\sum_{k=1}^sa_{jk})$.
Hence, we can calculate $p$th moments of the numerical solution as follows
\begin{align*}
\E[X_{n+1}^p+Y_{n+1}^p]\le&x_0^p\exp\left(\frac{pTC_\alpha}2\right)+y_0^p\exp\left(pC_A\right)\E\left[\exp\left(C_\beta pW(T)\right)\right]\\
=&x_0^p\exp\left(\frac{pTC_\alpha}2\right)+y_0^p\exp\left(pC_A\right)\sum_{k=0}^\infty\frac{(C_\beta p)^k\E[W(T)^k]}{k!}\\
\le&x_0^p\exp\left(\frac{pTC_\alpha}2\right)+y_0^p\exp\left(pC_A\right)\sum_{k=0}^\infty\frac{(C_\beta p)^{2k}(2k-1)!!T^k}{(2k)!}\\
\le&x_0^p\exp\left(\frac{pTC_\alpha}2\right)+y_0^p\exp\left(pC_A\right)\exp\left(C^2_\beta p^2T\right)
\end{align*}
since $\frac{(2k-1)!!}{(2k)!}=\frac1{(2k)!!}\le\frac1{k!}$.
\end{proof}

Based on the uniform boundedness above, we give the order conditions for \eqref{sto_RK_sim}.
\begin{tm}\label{tm_3}
Assume that methods \eqref{sto_RK_sim} satisfy the following conditions
\begin{align}\label{ordercondition_1}
\sum_{j=1}^s\alpha_j=\sum_{j=1}^s\beta_j=1.
\end{align}
Then these methods show global order one in $L^1(\Omega)$.
\end{tm}
\begin{proof}
Based on the positivity of the solution, we introduce the auxiliary coordinates $u=\ln x$ and $v=\ln y$. Then It\^o's formula yields
\begin{equation}\label{uv_sim}
\begin{aligned}
du=&\left(-e^v+1\right)dt,\\
dv=&\left(e^u-\frac32\right)dt+dW.
\end{aligned}
\end{equation}
Considering the local error between \eqref{uv_sim} and \eqref{sto_RK_sym} with the coefficients $\gamma_i,\eta_i,\sigma_i$, $i=1,2$, chosen as those in the beginning of Section 3.3, and utilizing condition \eqref{ordercondition_1}, we obtain
\begin{align*}
u(h)-U_1=\int_0^h\sum_{j=1}^s\alpha_je^{v_j}-e^{v(t)}dt,
\end{align*}
in which
\begin{align*}
\sum_{j=1}^s\alpha_je^{v_j}=&\sum_{j=1}^s\alpha_j\Bigg{[}e^{V_0}+e^{V_0}\bigg{(}h\sum_{k=1}^sa_{jk}\Big(e^{u_k}-\frac32\Big)+J_1\sum_{k=1}^sb_{jk}\bigg{)}+\\
&+\frac12e^{\theta V_0+(1-\theta)v_j}\bigg{(}h\sum_{k=1}^sa_{jk}\Big(e^{u_k}-\frac32\Big)+J_1\sum_{k=1}^sb_{jk}\bigg{)}^2\Bigg{]}
\end{align*}
with some $\theta\in(0,1)$, and
\begin{align*}
e^{v(t)}=&e^{V_0}+\int_0^te^{v(s)+u(s)}-e^{v(s)}ds+\int_0^te^{v(s)}dW(s).
\end{align*}
Thus, we can conclude from above and condition \eqref{ordercondition_1} that the expectation of the local deviation is of order 2:
\begin{align*}
&\big{|}\E(u(h)-U_1)\big{|}=\Bigg{|}\E\int_0^h\Bigg{[}he^{V_0}\sum_{j,k=1}^s\alpha_ja_{jk}\Big(e^{u_k}-\frac32\Big)-\int_0^t\left(e^{v(s)+u(s)}-e^{v(s)}\right)ds\\
+&\frac12e^{\theta V_0+(1-\theta)v_j}\bigg{(}h\sum_{j,k=1}^s\alpha_ja_{jk}\Big(e^{u_k}-\frac32\Big)+J_1\sum_{k=1}^sb_{jk}\bigg{)}^2\Bigg{]}dt\Bigg{|}
\le Ch^2,
\end{align*}
where we have used the uniform boundedness of solution $x(t)=e^{u(t)}$ and $y(t)=e^{v(t)}$.
Similarly, the local strong error has order of accuracy ${\rm 3/2}$:
\begin{align*}
\left(\E\left|u(h)-U_1\right|^2\right)^{\frac12}=&\Bigg{(}\E\Bigg{|}\int_0^h\Bigg{[}e^{\tilde\theta V_0+(1-\tilde\theta)v_j}\bigg{(}h\sum_{j,k=1}^s\alpha_ja_{jk}\Big(e^{u_k}-\frac32\Big)+J_1\sum_{j,k=1}^s\alpha_jb_{jk}\bigg{)}\\
&-\int_0^t\left(e^{v(s)+u(s)}-e^{v(s)}\right)ds-\int_0^te^{v(s)}dW(s)\Bigg{]}dt\Bigg{|}^2\Bigg{)}^{\frac12}
\le Ch^{\frac32}
\end{align*}
for $\tilde\theta\in(0,1)$, and so are the estimations of $v(h)-V_1$.
Based on the fundamental theorem on the mean-square order of convergence (see e.g. \cite{MT04}), one has immediately
\begin{align*}
\left(\E\left[\left|u(t_n)-U_n\right|^2+\left|v(t_n)-V_n\right|^2\right]\right)^\frac12\le Ch
\end{align*}
for any $t_n=nh\in[0,T]$,
which ensures the final result
{\small
\begin{align*}
&\E\left(\left|x(t_n)-X_n\right|+\left|y(t_n)-Y_n\right|\right)
=\E\left(\left|e^{u(t_n)}-e^{U_n}\right|+\left|e^{v(t_n)}-e^{V_n}\right|\right)\\
\le&\E\left(\left|e^{\theta_1 U_n+(1-\theta_1)u(t_n)}(u(t_n)-U_n)\right|+\left|e^{\theta_2 V_n+(1-\theta_2)v(t_n)}(v(t_n)-V_n)\right|\right)\\
\le&\max\left\{\left|e^{\theta_1 U_n+(1-\theta_1)u(t_n)}\right|_{L^2(\Omega)},\left|e^{\theta_2 V_n+(1-\theta_2)v(t_n)}\right|_{L^2(\Omega)}\right\}\left(\E\left[\left|u(t_n)-U_n\right|^2+\left|v(t_n)-V_n\right|^2\right]\right)^\frac12\\
\le&Ch
\end{align*}}
according to the uniform boundedness of the solution $(x(t),y(t))^{\top}$ and $(X_n,Y_n)^{\top}$ shown in Corollary \ref{cor3.1} and Theorem \ref{tm3.2} with $\theta_1,\theta_2\in(0,1)$.
\end{proof}

We can also get the order conditions for the KPRK methods in the same procedure.
\begin{tm}\label{KPRK2}
Assume that methods \eqref{sto_PRK} satisfy the following conditions
\begin{align*}
\sum_{j=1}^s\alpha_j=\sum_{j=1}^s\beta_j=1,~\sum_{j=1}^s\tilde{\alpha}_j=\sum_{j=1}^s\tilde{\beta}_j=1
\end{align*}
Then these methods show global order one in $L^1(\Omega)$.
\end{tm}

\begin{rk}
The procedure above is also available for high dimensions, i.e., $(x^\top,y^\top)^\top\in\R^{2d}$, by defining a product `$\bullet$' as
$$e\bullet f=f\bullet e:=(e_1f_1,\cdots,e_df_d)^\top\in\R^d,\quad \epsilon\bullet f=f\bullet \epsilon=\epsilon f$$
for any vectors $e=(e_1,\cdots,e_d)^\top,~f=(f_1,\cdots,f_d)^\top\in\R^d$ and constant $\epsilon\in\R$. Noticing that it satisfies that
$|e\bullet f|\le|e||f|$.
Also, let $\exp(f):=(\exp(f_1),\cdots,\exp(f_d))^\top$, and $\ln(f):=(\ln(f_1),\cdots,\ln(f_d))^\top$ if in addition $f\in\R^d_+$.
\end{rk}

\section{Some low-stage stochastic KRK methods}
In this section, both order conditions and stochastic K-symplectic conditions are used to find some low-stage specific stochastic K-symplectic methods for \eqref{LV}.
\begin{scheme}{One-stage stochastic KRK method}

We first consider one-stage stochastic KRK methods in the following form
\begin{center}
\begin{tabular}{c|c}
       & $a_{11}$ \quad $b_{11}$\\
\hline & $\alpha_1$ \quad $\beta_1$
\end{tabular}.
\end{center}
Using the order conditions \eqref{ordercondition_1} and stochastic K-symplectic conditions \eqref{condition}, the following results hold:
\begin{equation*}\label{midmethod}
\alpha_1=2a_{11},~\beta_1=2b_{11},~\alpha_1=1,~\beta_1=1,
\end{equation*}
which admits a unique solution $\alpha_1=1,\beta_1=1,a_{11}=1/2,b_{11}=1/2$. We get the unique
one-stage stochastic KRK method with Butcher tableau
\begin{center}
\begin{tabular}{c|c}
       & {\rm 1/2} \quad {\rm 1/2}\\
\hline & {\rm 1} \quad {\rm 1}
\end{tabular},
\end{center}
more precisely,
\begin{equation*}\label{midxy}
\begin{split}
    &X_{n+1}=X_n\exp\Big(-h\sqrt{Y_{n+1}Y_n}+h\Big),\\
    &Y_{n+1}=Y_n\exp\Big(h\sqrt{X_{n+1}X_n}-\frac{3h}{2}\Big)\exp(J_n),
\end{split}
\end{equation*}
which has strong global order one.
\end{scheme}

\begin{scheme}{Two-stage stochastic KRK method}

We consider particularly a class of two-stage stochastic KRK methods as follows
\begin{center}
\begin{tabular}{c|cccc}
       & $a_{11}$  &{\rm 0}  &$b_{11}$  &{\rm 0}\\
       & $a_{21}$  &$a_{22}$  &$b_{21}$  &$b_{22}$\\
\hline & $\alpha_1$  &$\alpha_2$ & $\beta_1$ &$\beta_2$
\end{tabular}
\end{center}
with conditions \eqref{ordercondition_1} and \eqref{condition} reading
\begin{equation*}
\begin{split}
  &\alpha_1+\alpha_2=1,~~\beta_1+\beta_2=1,\\[2mm]
  &\alpha_1=2a_{11},~~\beta_1=2b_{11},\\[2mm]
  &\alpha_2=2a_{22},~~\beta_2=2b_{22},\\[2mm]
  &a_{21}=\alpha_1,~~b_{21}=\beta_1.
\end{split}
\end{equation*}
Let $a_{11},b_{11}\in(0,1/2)$ be free parameters, we obtain a family of two-stage diagonally implicit stochastic KRK methods with strong global order one:
\begin{center}\label{method}
\begin{tabular}{c|cccc}
       & $a_{11}$  &{\rm 0}  &$b_{11}$  &{\rm 0}\\
       & $2a_{11}$  &{\rm 1/2}$-a_{11}$  &$2b_{11}$  &{\rm 1/2}$-b_{11}$\\
\hline & $2a_{11}$  &{\rm 1}$-2a_{11}$ & $2b_{11}$ &{\rm 1}$-2b_{11}$
\end{tabular}.
\end{center}
In particular, choosing $a_{11}=1/8$ and $b_{11}=1/4$, we get the scheme as
\begin{center}\label{method1}
\begin{tabular}{c|cccc}
       & {\rm 1/8}  &{\rm 0}  &{\rm 1/4}  &{\rm 0}\\
       & {\rm 1/4}  &{\rm 3/8}  &{\rm 1/2}  &{\rm 1/4}\\
\hline & {\rm 1/4}  &{\rm 3/4} & {\rm 1/2} &{\rm 1/2}
\end{tabular}.
\end{center}
\end{scheme}

Similarly, we can obtain the following first order KPRK methods expressed in Butcher tableaus under conditions in Theorem \ref{KPRK1} and \ref{KPRK2}.
\begin{scheme}{One-stage stochastic KPRK method}

Following is the unique one-stage stochastic KPRK method
\begin{center}\label{scheme_3}
\begin{tabular}{c|cccc}
       & {\rm 1/2}  &{\rm 1/2}  &{\rm 1/2}  &{\rm 1/2}\\
\hline & {\rm 1}    &{\rm 1}    &{\rm 1}    &{\rm 1}
\end{tabular},
\end{center}
which is the same as {\bf Scheme 1}.
\end{scheme}

\begin{scheme}{Two-stage stochastic KPRK method}

A specific choice of two-stage stochastic KPRK methods is
\begin{center}\label{LVmethod}
	\begin{tabular}{c|cccccccc}
		& {\rm 0}  &{\rm 0}  &{\rm 1/2}  &{\rm 0}& {\rm 0}  &{\rm 0}  &{\rm 1/2}  &{\rm 0}\\
		& {\rm 1/2}  &{\rm 1/2}  &{\rm 1/2}  &{\rm 0} &{\rm 1/2}  &{\rm 1/2}  &{\rm 1/2}  &{\rm 0}\\
		\hline & {\rm 1/2}  &{\rm 1/2} & {\rm 1/2} &{\rm 1/2} &{\rm 1/2}  &{\rm 1/2} & {\rm 1/2} &{\rm 1/2}
	\end{tabular}.
\end{center}
More precisely, it can be simplified as an explicit scheme:
\begin{equation*}\label{svxy}
	\begin{aligned}
		&y_1=Y_n\exp\left(\frac{h}2\left(X_n-\frac32\right)\right)\exp\left(\frac{J_n}2\right),\\
		&X_{n+1}=X_n\exp\left(-h y_1+h\right),\\
		&Y_{n+1}=Y_n\exp\left(\frac{h}2\left(X_{n+1}+X_n-3\right)\right)\exp\left(J_n\right),
	\end{aligned}
\end{equation*}
which is also of order one.
\end{scheme}

\section{Numerical experiments}\label{sec:num}

In this section, several numerical experiments are given in comparison with the following two widely used non-K-symplectic schemes applied to the LV model \eqref{LV}: 
Euler--Maruyama (EM) scheme
\begin{equation*}\label{euler}
\begin{aligned}
&X_{n+1}=X_n+h(-X_nY_n+X_n),\\
&Y_{n+1}=Y_n+h\left(X_nY_n-Y_n\right)+Y_n J_n
\end{aligned}
\end{equation*}
and Milstein scheme
\begin{equation*}\label{milstein}
\begin{aligned}
&X_{n+1}=X_n+h(-X_nY_n+X_n),\\
&Y_{n+1}=Y_n+h\left(X_nY_n-\frac{3}{2}Y_n\right)+Y_n J_n+\frac{1}{2} Y_n J_n^2.
\end{aligned}
\end{equation*}

Through out these experiments, the expectation is approximated by taking averaged value over 1000 realizations. We use the solution of {\bf Scheme 4} with a finer step size $2^{-12}$ as the reference value of the exact solution.
Firstly, the convergence order in $L^1(\Omega)$ as well as the convergence error is investigated for {\bf Scheme 1}, {\bf Scheme 4}, EM scheme and Milstein scheme.

\begin{figure}[th!]
	\centering
	\includegraphics[width=0.8\textwidth]{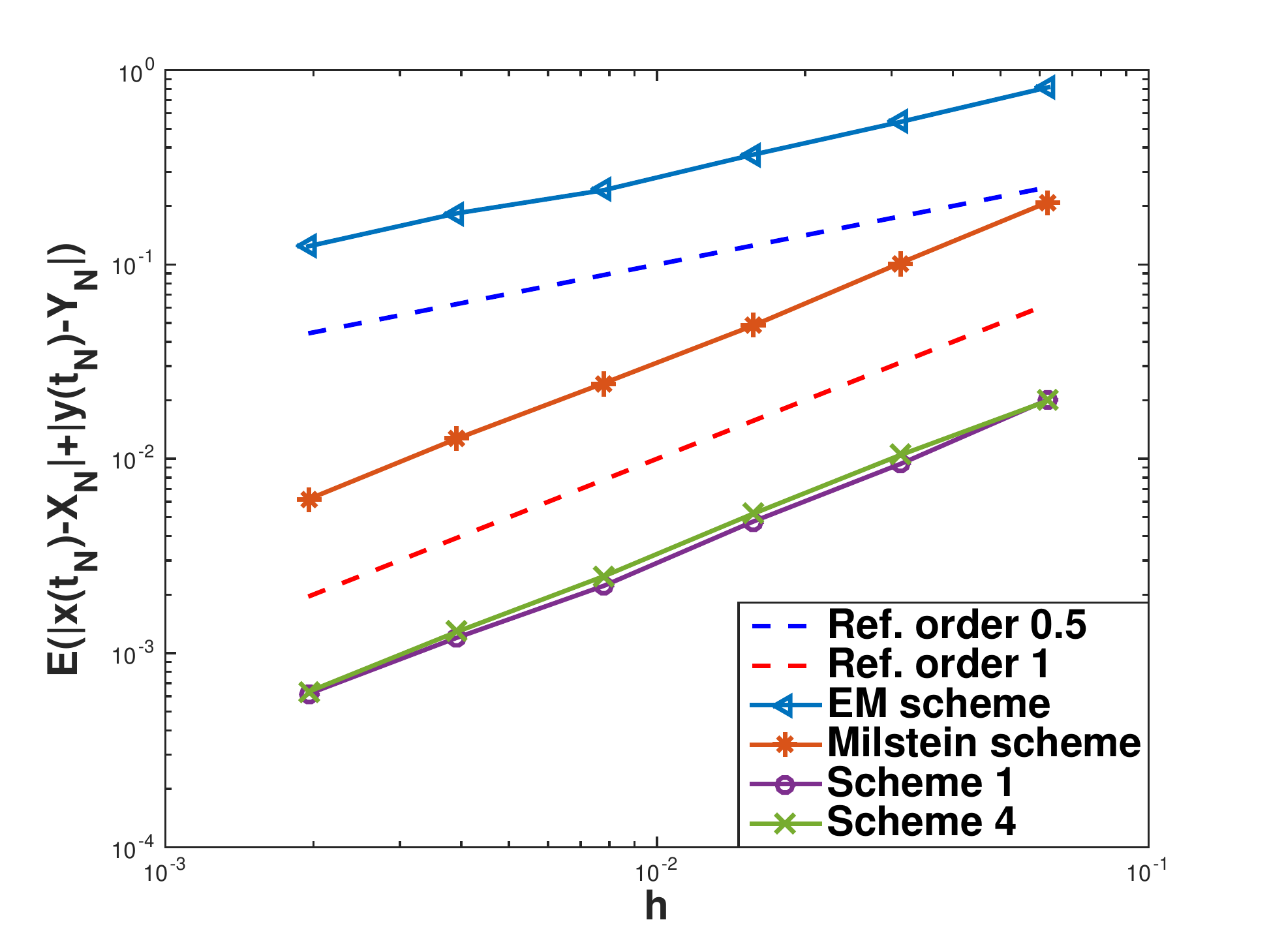}
	\caption{ $L^1(\Omega)$-convergence order for numerical solution $(X_N,Y_N)$ at $T=1$ in log-log scale with step size $h\in\{2^{-i},4\le i\le9\}$}
	\label{fig1}
\end{figure}

\begin{table}[th!]
\centering
\caption{$L^1(\Omega)$ convergence error for schemes with $h=2^{-6}$}
\begin{tabular}{c|ccccccc}
	\hline
	$T$  & 0.5 &  1 & 5 & 10 & 20   \\
	\hline
	EM & 5.18e-01&3.93e-01& 2.35e-01& 1.89& 3.39\\
Milstein & 5.20e-02 & 4.99e-02 & 1.52e-01 &1.26  & 2.91 \\
{\bf Scheme 1}  & 6.80e-03 & 5.00e-03 & 1.74e-02 & 8.08e-02  & 4.82e-01  \\
{\bf Scheme 4} & 7.00e-03 & 5.20e-03 & 1.67e-02 & 1.08e-01  & 7.67e-01 \\

	\hline
\end{tabular}
\label{tab}
\end{table}

It can be observed from Figure \ref{fig1} that EM scheme is of order 0.5 while the other three schemes are all of order 1 in $L^1(\Omega)$ compared with the reference lines. However, the convergence errors for {\bf Scheme 1} and {\bf Scheme 4} are smaller than that of EM scheme and Milstein scheme. To make it clearer, we give the convergence error of these schemes for different time intervals $T=0.5,1,5,10,20$ in Table \ref{tab}. It shows that {\bf Scheme 1} and {\bf Scheme 4} are more stable than EM scheme and Milstein scheme, which also indicates good performances of K-symplectic schemes over long time.

\begin{figure}[th!]
	\centering
	\subfigure[\bf Scheme 1]{
		\begin{minipage}[t]{0.45\linewidth}
			\includegraphics[width=1.1\textwidth]{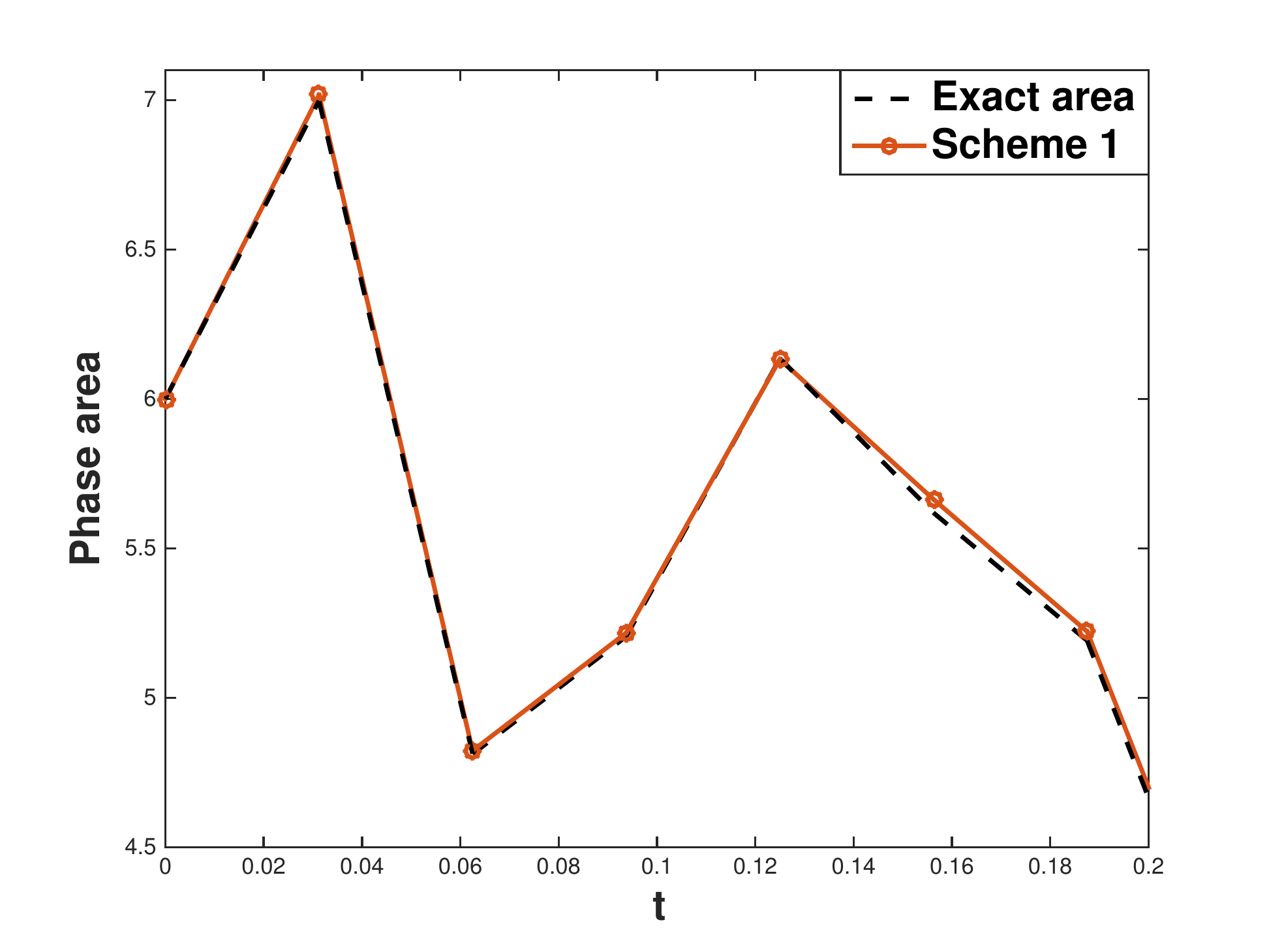}
		\end{minipage}
	}
	\subfigure[Milstein scheme]{
		\begin{minipage}[t]{0.45\linewidth}
			\includegraphics[width=1.1\textwidth]{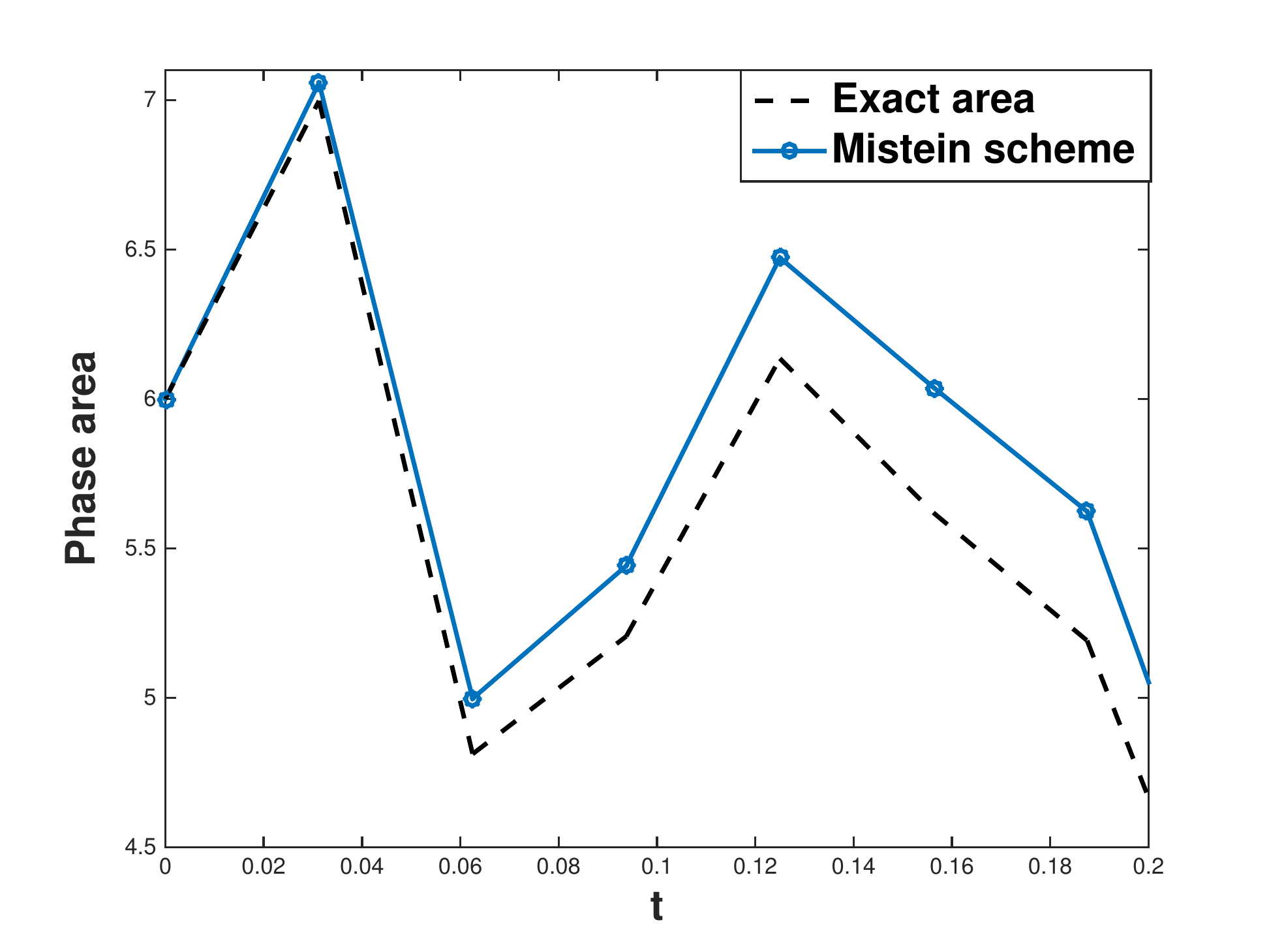}
		\end{minipage}
	}
	\caption{Evolution of the phase area for {\bf Scheme 1} and Milstein scheme ($h=2^{-5},T=0.2$)}
	\label{fig2}
\end{figure}

\begin{figure}[th!]
	\centering
	\includegraphics[width=0.8\textwidth]{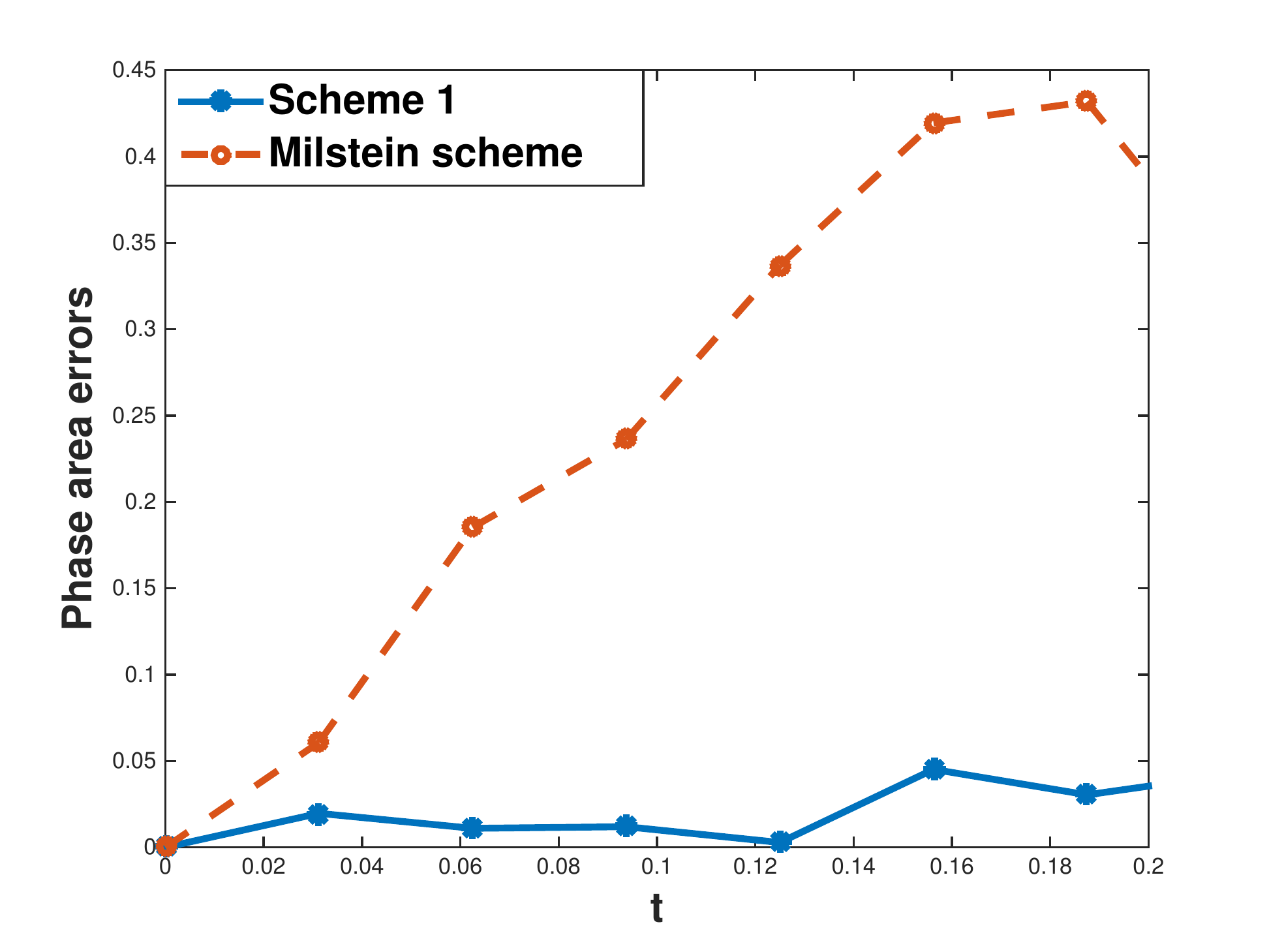}
	\caption{ Evolution of the phase area error for {\bf Scheme 1} and  Milstein scheme ($h=2^{-5},$ $T=0.2$)}
	\label{fig3}
\end{figure}

Next, we show the phase area evolution and the error of the phase area for {\bf Scheme 1} and Milstein scheme which are all of order 1. We choose a triangle determined by three points $w^1_0=(1,7), w^2_0=(7,1)$ and $w^3_0=(2,8)$ as the initial area. We can get three family of points $\{w_n^i\}_{n\ge0}, i=1,2,3,$ under the propagation of a specific scheme. In a small time interval $[0,0.2]$, we regard the phase area at step $n$ as the triangle determined by points $\{w_n^1,w_n^2,w_n^3\}$ since the approximate error is small enough during small interval $[0,0.2]$.

Figure \ref{fig2} shows the evolution of the phase area. The evolution of the phase area of {\bf Scheme 1} is almost the same as the exact one, while the phase area of the Milstein scheme turns to deviate from the exact one. This phenomenon appears more evident in Figure \ref{fig3}, where we simulate the error of the phase area for {\bf Scheme 1} and the Milstein scheme. The good performance of {\bf Scheme 1} benefits from the preservation of the geometric structure, which shows the superiority of K-symplectic schemes.


\section*{Acknowledgement}

J. Hong and X. Wang are supported by the National Natural Science Foundation of China (No. 91530118, No. 91130003, No. 11021101, No. 91630312 and No. 11290142). L. Ji is supported by the National Natural Science Foundation of China (No. 11601032, No. 11471310). J. Zhang is supported by  the National Natural Science Foundation of China (No. 11761033).

\section*{Appendix}
\begin{proof}[\bf Proof of Theorem \ref{tm_1}]
Since the relation
\begin{equation*}
  \left[\frac{\partial \varphi_t(z_0)}{\partial z_0}\right]^{\top}K(\varphi_t(z_0))\left[\frac{\partial \varphi_t(z_0)}{\partial z_0}\right]=K(z_0)
\end{equation*}
is satisfied for $t=0$ ($\varphi_0$ is the identity map). Then, the equality \eqref{stochastic_k_symplectic} is fulfilled if and only if
\begin{equation*}
  {\rm d}_t\left(\left[\frac{\partial \varphi_t(z_0)}{\partial z_0}\right]^{\top}K(\varphi_t(z_0))\left[\frac{\partial \varphi_t(z_0)}{\partial z_0}\right]\right)=0,\quad a.s.
\end{equation*}
Therefore, we obtain
\begin{align*}
{\rm d}_t&\left(\left[\frac{\partial \varphi_t(z_0)}{\partial z_0}\right]^{\top}K(\varphi_t(z_0))\left[\frac{\partial \varphi_t(z_0)}{\partial z_0}\right]\right)={\rm d}_t\left(\left[\frac{\partial \varphi_t(z_0)}{\partial z_0}\right]^{\top}\right)K(\varphi_t(z_0))\left[\frac{\partial \varphi_t(z_0)}{\partial z_0}\right]\nonumber\\
&+\left[\frac{\partial \varphi_t(z_0)}{\partial z_0}\right]^{\top}{\rm d}_t\left(K(\varphi_t(z_0))\right)\left[\frac{\partial \varphi_t(z_0)}{\partial z_0}\right]+\left[\frac{\partial \varphi_t(z_0)}{\partial z_0}\right]^{\top}K(\varphi_t(z_0)){\rm d}_t\left(\left[\frac{\partial \varphi_t(z_0)}{\partial z_0}\right]\right)\nonumber\\
:=&I+II+III.
\end{align*}
For term $I$, it yields
\begin{equation*}
  \begin{split}
    I=&\frac{\partial}{\partial z_0}\Bigg(K^{-1}(\varphi_t(z_0))\nabla H_0(\varphi_t(z_0)){\rm d}t+K^{-1}(\varphi_t(z_0))\nabla H_1(\varphi_t(z_0))\circ {\rm d}W\Bigg)^{\top}K(\varphi_t(z_0))\left[\frac{\partial \varphi_t(z_0)}{\partial z_0}\right]\\
    =& \left[\frac{\partial \varphi_t(z_0)}{\partial z_0}\right]^{\top}\frac{\partial (K^{-1}(z)\nabla H_0(z){\rm d}t+K^{-1}(z)\nabla H_1(z)\circ {\rm d}W)}{\partial z}\Bigg|_{z=\varphi_t(z_0)}^{\top}K(\varphi_t(z_0))\left[\frac{\partial \varphi_t(z_0)}{\partial z_0}\right].
  \end{split}
\end{equation*}
Similarly, for term $III$
\begin{equation*}
  \begin{split}
    III=\left[\frac{\partial \varphi_t(z_0)}{\partial z_0}\right]^{\top}K(\varphi_t(z_0))\frac{\partial (K^{-1}(z)\nabla H_0(z){\rm d}t+K^{-1}(z)\nabla H_1(z)\circ {\rm d}W)}{\partial z}\Bigg|_{z=\varphi_t(z_0)}\left[\frac{\partial \varphi_t(z_0)}{\partial z_0}\right].
  \end{split}
\end{equation*}

Combining with these two equalities, we can obtain
\begin{align*}
&{\rm d}_t\Bigg(\left[\frac{\partial \varphi_t(z_0)}{\partial z_0}\right]^{\top}K(\varphi_t(z_0))\left[\frac{\partial \varphi_t(z_0)}{\partial z_0}\right]\Bigg)\\
=&\left[\frac{\partial \varphi_t(z_0)}{\partial z_0}\right]^{\top}\Bigg(\frac{\partial (K^{-1}(z)\nabla H_0(z))}{\partial z}\Bigg|_{z=\varphi_t(z_0)}^{\top}K(\varphi_t)
    +\frac{\partial K}{\partial z}(K^{-1}(z)\nabla H_0(z))\Bigg|_{z=\varphi_t(z_0)}\\
    &+K(\varphi_t)\frac{\partial (K^{-1}(z)\nabla H_0(z))}{\partial z}\Bigg|_{z=\varphi_t(z_0)}\Bigg)\left[\frac{\partial \varphi_t(z_0)}{\partial z_0}\right]{\rm d}t\\
   & +\left[\frac{\partial \varphi_t(z_0)}{\partial z_0}\right]^{\top}
\Bigg(\frac{\partial (K^{-1}(z)\nabla H_1(z))}{\partial z}\Bigg|_{z=\varphi_t(z_0)}^{\top}K(\varphi_t)
    +\frac{\partial K}{\partial z}(K^{-1}(z)\nabla H_1(z))\Bigg|_{z=\varphi_t(z_0)}\\
  &  +K(\varphi_t)\frac{\partial (K^{-1}(z)\nabla H_1(z))}{\partial z}\Bigg|_{z=\varphi_t(z_0)}\Bigg)\left[\frac{\partial \varphi_t(z_0)}{\partial z_0}\right]\circ{\rm d}W(t)\\
=&:\left[\frac{\partial \varphi_t(z_0)}{\partial z_0}\right]^{\top}A\left[\frac{\partial \varphi_t(z_0)}{\partial z_0}\right]{\rm d}t+\left[\frac{\partial \varphi_t(z_0)}{\partial z_0}\right]^{\top}B\left[\frac{\partial \varphi_t(z_0)}{\partial z_0}\right]\circ{\rm d}W(t).
\end{align*}
For convenience, let $\tilde{k}_{ij},i,j=1,2,\cdots,2d$ be the elements of $K^{-1}$. Since $K^{-1}(z)K(z)=I_{2d}$, we have
\begin{equation*}
\begin{split}
\sum_{j=1}^{2d}\tilde{k}_{ij}k_{jk}=\delta_{ik},~~&{\rm for~all}~~i,k=1,2,\cdots,2d,\\
\sum_{j=1}^{2d}\left(\frac{\partial \tilde{k}_{ij}}{\partial z_l}k_{jk}+\tilde{k}_{ij}\frac{\partial k_{jk}}{\partial z_l}\right)=0,~~&{\rm for~all}~~i,k,l=1,2,\cdots,2d.
\end{split}
\end{equation*}
Now, we consider the $(i,j)$th element $A_{ij}$ of matrix $A$, it holds
\begin{align*}
    A_{ij}&=\sum_{r=1}^{2d}\left[\frac{\partial (K^{-1}\nabla H_0)_r}{\partial z_i}k_{rj}+\frac{\partial k_{ij}}{\partial z_r}(K^{-1}\nabla H_0)_r+k_{ir}\frac{\partial (K^{-1}\nabla H_0)_r}{\partial z_j}\right]\\[3mm]
    &=\sum_{r,s=1}^{2d}\left[\left(\frac{\partial \tilde{k}_{rs}}{\partial z_i}\frac{\partial H_0}{\partial z_s}+\tilde{k}_{rs}\frac{\partial^2 H_0}{\partial z_i\partial z_s}\right)k_{rj}+\tilde{k}_{rs}\frac{\partial k_{ij}}{\partial z_r}\frac{\partial H_0}{\partial z_s}+k_{ir}\left(\frac{\partial \tilde{k}_{rs}}{\partial z_j}\frac{\partial H_0}{\partial z_s}+\tilde{k}_{rs}\frac{\partial^2 H_0}{\partial z_j\partial z_s}\right)\right]\\[3mm]
    &=\sum_{r,s=1}^{2d}\left(k_{rj}\frac{\partial \tilde{k}_{rs}}{\partial z_i}+\tilde{k}_{rs}\frac{\partial k_{ij}}{\partial z_r}+k_{ir}\frac{\partial \tilde{k}_{rs}}{\partial z_j}\right)\frac{\partial H_0}{\partial z_s}+\sum_{s=1}^{2d}\left(\delta_{is}\frac{\partial^2 H_0}{\partial z_j\partial z_s}-\delta_{sj}\frac{\partial^2 H_0}{\partial z_i\partial z_s}\right)\\[3mm]
    &=\sum_{r,s=1}^{2d}\left(-\tilde{k}_{rs}\frac{\partial k_{rj}}{\partial z_i}+\tilde{k}_{rs}\frac{\partial k_{ij}}{\partial z_r}-\tilde{k}_{rs}\frac{\partial k_{ir}}{\partial z_j}\right)\frac{\partial H_0}{\partial z_s}+\sum_{s=1}^{2d}\left(\delta_{is}\frac{\partial^2 H_0}{\partial z_j\partial z_s}-\delta_{sj}\frac{\partial^2 H_0}{\partial z_i\partial z_s}\right)\\[3mm]
    &=\sum_{r,s=1}^{2d}\left(\frac{\partial k_{jr}}{\partial z_i}+\tilde{k}_{rs}\frac{\partial k_{ij}}{\partial z_r}+\frac{\partial k_{ri}}{\partial z_j}\right)\tilde{k}_{rs}\frac{\partial H_0}{\partial z_s}+\sum_{s=1}^{2d}\left(\delta_{is}\frac{\partial^2 H_0}{\partial z_j\partial z_s}-\delta_{sj}\frac{\partial^2 H_0}{\partial z_i\partial z_s}\right)\\[3mm]
    &=\sum_{s=1}^{2d}\left(\delta_{is}\frac{\partial^2 H_0}{\partial z_j\partial z_s}-\delta_{sj}\frac{\partial^2 H_0}{\partial z_i\partial z_s}\right),
\end{align*}
where the equality is due to the Jacobi identity \eqref{Jacobi_identity}. Similarly, we have
\begin{equation*}
  \begin{split}
    B_{ij}&=\sum_{s=1}^{2d}\left(\delta_{is}\frac{\partial^2 H_1}{\partial z_j\partial z_s}-\delta_{sj}\frac{\partial^2 H_1}{\partial z_i\partial z_s}\right),
  \end{split}
\end{equation*}
Thus, we have
\begin{equation*}
  \begin{split}
    {\rm d}_t&\Bigg(\left[\frac{\partial \varphi_t(z_0)}{\partial z_0}\right]^{\top}K(\varphi_t(z_0))\left[\frac{\partial \varphi_t(z_0)}{\partial z_0}\right]\Bigg)\\[3mm]
    =&\sum_{i,j}^{2d}\Big(A_{i,j}{\rm d}t+B_{ij}\circ{\rm d}W(t)\Big)\frac{\partial \varphi_t(z_0)}{\partial z_{0_i}}\frac{\partial \varphi_t(z_0)}{\partial z_0{_{j}}}\\[3mm]
    =&\sum_{i,j,s}^{2d}\left(\delta_{is}\frac{\partial^2 H_0}{\partial z_j\partial z_s}-\delta_{sj}\frac{\partial^2 H_0}{\partial z_i\partial z_s}\right)\frac{\partial \varphi_t(z_0)}{\partial z_{0_i}}\frac{\partial \varphi_t(z_0)}{\partial z_0{_{j}}}{\rm d}t\\[3mm]
    &+\sum_{i,j,s}^{2d}\left(\delta_{is}\frac{\partial^2 H_1}{\partial z_j\partial z_s}-\delta_{sj}\frac{\partial^2 H_1}{\partial z_i\partial z_s}\right)\frac{\partial \varphi_t(z_0)}{\partial z_{0_i}}\frac{\partial \varphi_t(z_0)}{\partial z_0{_{j}}}\circ{\rm d}W(t)\\[3mm]
    =&0.
  \end{split}
\end{equation*}
Then, the proof is completed.
\end{proof}

\begin{proof}[\bf Proof of Theorem \ref{tm_2}]
We have
\begin{equation*}
  \begin{split}
    {\rm d}_t\Big({\rm d}z\wedge K(z(t)){\rm d}z\Big)&=\sum_{i,j=1}^{2d}\Big({\rm d}({\rm d}_tz_i)\wedge k_{ij}{\rm d}z_j+{\rm d}z_i\wedge ({\rm d}_tk_{ij}){\rm d}z_{j}+{\rm d}z_i\wedge k_{ij}{\rm d}({\rm d}_tz_j)\Big)\\[2mm]
    &:=I+II+III.
  \end{split}
\end{equation*}
For terms $I$ and $III$, we get
\begin{equation*}
\begin{split}
  I=&\sum_{i,j,r,s}^{2d}k_{ij}\left[\left(\frac{\partial \tilde{k}_{ir}}{\partial z_s}\frac{\partial H_0}{\partial z_r}+\tilde{k}_{ir}\frac{\partial^2 H_0}{\partial z_s\partial z_r}\right){\rm d}t+\left(\frac{\partial \tilde{k}_{ir}}{\partial z_s}\frac{\partial H_1}{\partial z_r}+\tilde{k}_{ir}\frac{\partial^2 H_1}{\partial z_s\partial z_r}\right)\circ{\rm d}W(t)\right]{\rm d}z_s\wedge{\rm d}z_j\\[2mm]
  =&\sum_{i,j,r,s}^{2d}\left(k_{sj}\frac{\partial \tilde{k}_{sr}}{\partial z_i}\frac{\partial H_0}{\partial z_r}{\rm d}t+k_{sj}\frac{\partial \tilde{k}_{sr}}{\partial z_i}\frac{\partial H_1}{\partial z_r}\circ{\rm d}W(t)\right){\rm d}z_i\wedge{\rm d}z_j\\
  &-\sum_{j,r,s}^{2d}\delta_{jr}\left(\frac{\partial^2 H_0}{\partial z_s\partial z_r}{\rm d}t+\frac{\partial^2 H_1}{\partial z_s\partial z_r}\circ{\rm d}W(t)\right){\rm d}z_s\wedge{\rm d}z_j
\end{split}
\end{equation*}
and
\begin{equation*}
\begin{split}
  III=&\sum_{i,j,r,s}^{2d}k_{ij}\left[\left(\frac{\partial \tilde{k}_{jr}}{\partial z_s}\frac{\partial H_0}{\partial z_r}+\tilde{k}_{jr}\frac{\partial^2 H_0}{\partial z_s\partial z_r}\right){\rm d}t+\left(\frac{\partial \tilde{k}_{jr}}{\partial z_s}\frac{\partial H_1}{\partial z_r}+\tilde{k}_{jr}\frac{\partial^2 H_1}{\partial z_s\partial z_r}\right)\circ{\rm d}W(t)\right]{\rm d}z_i\wedge{\rm d}z_s\\[2mm]
  =&\sum_{i,j,r,s}^{2d}\left(k_{is}\frac{\partial \tilde{k}_{sr}}{\partial z_j}\frac{\partial H_0}{\partial z_r}{\rm d}t+k_{is}\frac{\partial \tilde{k}_{sr}}{\partial z_j}\frac{\partial H_1}{\partial z_r}\circ{\rm d}W(t)\right){\rm d}z_i\wedge{\rm d}z_j\\
  &+\sum_{i,r,s}^{2d}\delta_{ir}\left(\frac{\partial^2 H_0}{\partial z_s\partial z_r}{\rm d}t+\frac{\partial^2 H_1}{\partial z_s\partial z_r}\circ{\rm d}W(t)\right){\rm d}z_i\wedge{\rm d}z_s.
\end{split}
\end{equation*}
For the second term, it holds
\begin{equation*}
\begin{split}
  II&=\sum_{i,j,r,s}^{2d}\tilde{k}_{rs}\frac{\partial k_{ij}}{\partial z_r}\frac{\partial H_0}{\partial z_s}{\rm d}z_i\wedge{\rm d}z_j{\rm d}t+\sum_{i,j,r,s}^{2d}\tilde{k}_{rs}\frac{\partial k_{ij}}{\partial z_r}\frac{\partial H_1}{\partial z_s}{\rm d}z_i\wedge{\rm d}z_j\circ{\rm d}W(t)\\[2mm]
  &=\sum_{i,j,r,s}^{2d}\tilde{k}_{sr}\frac{\partial k_{ij}}{\partial z_s}\frac{\partial H_0}{\partial z_r}{\rm d}z_i\wedge{\rm d}z_j{\rm d}t+\sum_{i,j,r,s}^{2d}\tilde{k}_{sr}\frac{\partial k_{ij}}{\partial z_s}\frac{\partial H_1}{\partial z_r}{\rm d}z_i\wedge{\rm d}z_j\circ{\rm d}W(t).
\end{split}
\end{equation*}
Adding these three equalities, we get
\begin{equation*}
  \begin{split}
   & {\rm d}_t\Big({\rm d}z\wedge K(z(t)){\rm d}z\Big)\\
    =&\sum_{i,j,r,s=1}^{2d}\left(k_{sj}\frac{\partial \tilde{k}_{sr}}{\partial z_i}+\tilde{k}_{sr}\frac{\partial k_{ij}}{\partial z_s}+k_{is}\frac{\partial \tilde{k}_{sr}}{\partial z_j}\right)\frac{\partial H_0}{\partial z_r}{\rm d}z_i\wedge{\rm d}z_j{\rm d}t\\
    &+\sum_{i,j,r,s=1}^{2d}\left(k_{sj}\frac{\partial \tilde{k}_{sr}}{\partial z_i}+\tilde{k}_{sr}\frac{\partial k_{ij}}{\partial z_s}+k_{is}\frac{\partial \tilde{k}_{sr}}{\partial z_j}\right)\frac{\partial H_1}{\partial z_r}{\rm d}z_i\wedge{\rm d}z_j\circ{\rm d}W(t)\\
    &+2\sum_{i,r,s}^{2d}\delta_{ir}\left(\frac{\partial^2 H_0}{\partial z_s\partial z_r}{\rm d}t+\frac{\partial^2 H_1}{\partial z_s\partial z_r}\circ{\rm d}W(t)\right){\rm d}z_i\wedge{\rm d}z_s\\
    =&\sum_{i,j,r,s=1}^{2d}\left(\frac{\partial k_{js}}{\partial z_i}+\frac{\partial k_{ij}}{\partial z_s}+\frac{\partial k_{si}}{\partial z_j}\right)\tilde{k}_{sr}\frac{\partial H_0}{\partial z_r}{\rm d}z_i\wedge{\rm d}z_j{\rm d}t\\
    &+\sum_{i,j,r,s=1}^{2d}\left(\frac{\partial k_{js}}{\partial z_i}+\frac{\partial k_{ij}}{\partial z_s}+\frac{\partial k_{si}}{\partial z_j}\right)\tilde{k}_{sr}\frac{\partial H_1}{\partial z_r}{\rm d}z_i\wedge{\rm d}z_j\circ{\rm d}W(t)\\
    &+2\sum_{r,s}^{2d}\left(\frac{\partial^2 H_0}{\partial z_s\partial z_r}{\rm d}z_r\wedge{\rm d}z_s\right){\rm d}t+2\sum_{r,s}^{2d}\left(\frac{\partial^2 H_1}{\partial z_s\partial z_r}{\rm d}z_r\wedge{\rm d}z_s\right)\circ{\rm d}W(t)\\
    =&0,
  \end{split}
\end{equation*}
where the equality is due to the Jacobi identity \eqref{Jacobi_identity} and wedge property ${\rm d}z\wedge A{\rm d}z=0$ when $A$ is a symmetric matrix. Therefore, it follows from
\begin{equation*}
  \Big({\rm d}z\wedge K(z(t)){\rm d}z\Big)\Big|_{t=0}={\rm d}z_0\wedge K(z_0){\rm d}z_0,
\end{equation*}
we can obtain equality \eqref{wedge_theorem}.
\end{proof}


\bibliography{reference}
\bibliographystyle{plain}

\end{document}